\documentclass[12pt]{amsart}
\usepackage{ amsmath, amsthm, amsfonts,amssymb, hyperref}
\usepackage{multicol,graphicx}
\usepackage{color}
\usepackage{ifpdf}

\makeatletter \@addtoreset{equation}{section}

\makeatother

\def\fnum{equation}
\newtheorem{thm}[\fnum]{Theorem}
\newtheorem{cor}[\fnum]{Corollary}

\newtheorem{lem}[\fnum]{Lemma}

\theoremstyle{definition}
\newtheorem{defn}[\fnum]{Definition}

\newtheorem{rem}[\fnum]{Remark}

\numberwithin{equation}{section}

\newcommand{\dist}{{\text {dist}}}

\newcommand{\Length}{{\text {Length}}}
\newcommand{\Energy}{{\text {Energy}}}

\def\SS{{\bold  S}}

\newcommand{\eqr}[1]{(\ref{#1})}

\begin{document}

\title[Closed geodesics in Alexandrov spaces of curvature $\leq \,K$] {Closed geodesics in Alexandrov spaces of curvature bounded from above}
\author{Longzhi Lin}
\address{Department of Mathematics\\
Johns Hopkins University\\
3400 N. Charles St.\\
Baltimore, MD 21218}


\email{lzlin@math.jhu.edu}

\maketitle {}

\begin{abstract}
In this paper, we show a local energy convexity of $W^{1,2}$ maps
into $CAT(K)$ spaces. This energy convexity allows us to extend
Colding and Minicozzi's width-sweepout construction to produce
closed geodesics in any closed Alexandrov space of curvature bounded
from above, which also provides a generalized version of the Birkhoff-Lyusternik theorem on the existence of non-trivial closed geodesics in the Alexandrov setting.
\end{abstract}

\section{Introduction}
Closed geodesics have been investigated mainly in the case of closed (i.e., compact and without boundary) Riemannian manifolds, while various results were obtained for Finsler manifolds and in the more general case of metric spaces with certain special properties (known as Busemann $G$-spaces, see \cite{Bu}). These studies were initiated by Hadamard \cite{Had}, Poincar$\acute{\text{e}}$ \cite{Po} and Birkhoff \cite{B1}. The method of finding non-trivial closed geodesics on simply-connected manifolds by curve shortening map and  sweepouts goes back to Birkhoff in
1917: one pulls each curve in a sweepout (see Definition \ref{79})
on the manifold as tight as possible, in a continuous way and
preserving the sweepout. See \cite{B1},\cite{B2}, \cite{Cr} and \cite{CC} (page 533) for more about Birkhoff's ideas.\vskip 1mm

In this paper, we show the following local energy convexity of
$W^{1,2}$ maps into a $CAT(K)$ space (also known as an
$\mathfrak{R}_K$ domain, see \cite{BBI} or subsection \ref{63}
below).

\begin{thm} \label{24} Let $\Sigma$ be a compact Riemannian domain and
$(X,d\,)$ be an Alexandrov space of curvature bounded from above by
$K$. In an $\mathfrak{R}_K$ domain of $x\in X$, there exists
$\,\rho=\rho(x,K)>0\,$ such that for $u,v\in W^{1,2}(\Sigma,X)$ with
images staying in $B_{\rho}(x)\subset \mathfrak{R}_K,$ the following
holds:
\begin{equation} \label{74}
\frac{1}{4}\int_{\Sigma}|\nabla d\,(u,v)|^2\leq E^{u}+E^{v}-2E^{w}.
\end{equation}
Here $B_{\rho}(x)$ is the geodesic ball centered at $x$ with radius
$\rho\,$, $w=\frac{u+v}{2}$ is the mid-point map and $E$ is the
$2$-energy of maps into metric spaces (see subsection \ref{30})\,.
\end{thm}

We shall remark that Theorem \ref{24} provides a stronger (quantitative) version of a result of Burago, Burago and Ivanov \cite[Proposition 9.1.17]{BBI}. Theorem \ref{24} allows us to use Colding and Minicozzi's
width-sweepout construction of closed geodesics to produce closed
geodesics in another class of general metric spaces, namely, the (closed) Alexandrov spaces of curvature bounded from above. Our main
result is the following useful property of good sweepouts in any
closed Alexandrov space of curvature bounded from above:  {\textit{each curve in the tightened
sweepout whose length is close to the length of the longest curve in
the sweepout must itself be close to a closed geodesic}}; see
Theorem \ref{mainThm} below and cf. \cite{CM1}-\cite{CM3},
\cite{LW}, proposition 3.1 of \cite{CD}, proposition 3.1 of
\cite{Pi}, and 12.5 of \cite{Al}. Moreover, as an immediate corollary of the existence of good sweepouts we present in section \ref{Goods}, we obtain a generalized Birkhoff-Lyusternik theorem on the existence of non-trivial closed geodesics, cf. \cite{LyS}, \cite{Ly}.

\begin{thm} (Generalized Birkhoff-Lyusternik theorem) \label{BLthm} Let $(X,d)$ be a closed Alexandrov space of curvature bounded from above by
$K$. Suppose that the $k_0$-th homology group $H_{k_0}(X)$ is nonzero for some $k_0\geq 1$, then $(X,d)$ admits at least one non-trivial closed geodesic.
\end{thm}

In \cite{CM1}, Colding and Minicozzi introduced the geometric
invariant of closed Riemannian manifolds that they call the $width$.
They succeeded in using the local linear replacement as Birkhoff's
curve shortening map to construct good sweepouts that produce at
least one closed geodesic, which realizes the width as its energy.
In particular, there exist closed geodesics on any closed Riemannian
manifold. The argument only produces non-trivial closed geodesics
when the width is positive (see footnote \ref{59}). Their local
linear replacement process is a discrete gradient flow (for the
length functional of curves), and it depends solely on a local
energy convexity for $W^{1,2}$ maps into closed Riemannian manifolds
(see Lemma 4.2 of \cite{CM1}), which controls the distance of curves
in the tightened sweepout from closed geodesics explicitly. (See
\cite{CM2} for $2$-width of manifolds and sweepouts by $2$-spheres
instead of circles to produce minimal surfaces on manifolds.) \vskip
1mm

Recall that any compact
smooth Riemannian manifold is an Alexandrov space of curvature
bounded from above by some $K$ (see Theorem \ref{78}). As mentioned, the local energy convexity is crucial in the
width-sweepout construction of closed geodesics. Therefore it is reasonable that if we can find a similar energy convexity for maps into
$CAT(K)$ spaces, we can extend Colding and Minicozzi's width-sweepout
construction to produce closed geodesics in any closed Alexandrov space of curvature bounded from above by $K$,  which is locally a $CAT(K)$ space. In theorem 2.2 of \cite{KS} (see equation (2.2iv) of \cite{KS}), Korevaar and Schoen provided an energy
convexity of $W^{1,2}$ maps from a compact Riemannian domain into an
NPC space (when $K=0$) for which the model space{\footnote{Also known as the $K$-plane, see footnote \ref{64}.}} is $\mathbf{R}^2$. Since the model
space for the case of $K>0$ is the standard Euclidean 2-hemisphere $\mathbf{S}_K$,
which is locally $\mathbf{R}^2$, it is perhaps not surprising
that a similar energy convexity should also hold locally. In fact, using the $K$-quadrilateral cosine $cosq_K$
that Berg and Nikolaev defined in \cite{BN} we will be able to show
that, with a small image assumption, the same (up to a constant)
energy convexity still holds for $W^{1,2}$ maps into a $CAT(K)$
space with $K>0$. Namely, we have Theorem \ref{24} above (see
section \ref{38} for the proof). \vskip 1mm

We remark that in \cite{Se} (unpublished) Serbinowski provides a
local energy convexity for maps into a $CAT(K)$ space. His result is
an analogue to (\ref{74}) and leads to the uniqueness of the
small-image solution to the Dirichlet problem from a Riemannian domain
into an Alexandrov space of curvature $\leq K$ , yet we cannot use his version of convexity to control the
distance between curves directly and explicitly.\vskip 2mm

The paper is organized as follows. In section \ref{Pre} we collect some basic definitions and properties of the Alexandrov spaces of curvature bounded from above and also the energy of maps into metric spaces introduced by Korevaar and Schoen \cite{KS}. In section \ref{38}, we will prove the local energy convexity in $CAT(K)$ spaces and we postpone the elementary proof of the main Lemma \ref{34} in Appendix \ref{36}. In section \ref{Goods}, we extend the delicate construction of good sweepouts used by Colding and Minicozzi \cite{CM1} to our Alexandrov setting, while we show this sweepout construction satisfies the properties of the Birkhoff curve-shorting process in Appendixes \ref{51} and \ref{75}. Finally we give a proof of Theorem \ref{BLthm} in section \ref{GBL}.

\subsection*{\textit{Acknowledgement}} The author would like to thank
Professor William Minicozzi for his continued guidance. The author would also like to thank the referee for the suggestion on adding Theorem \ref{BLthm} to this paper.

\section{Preliminaries} \label{Pre}
\subsection{Alexandrov space of curvature $\leq K$} \label{63} In the 1950's, Alexandrov introduced spaces of
curvature bounded from above in his papers \cite{A1}, \cite{A2}. The
terminology $CAT(K)$ spaces was then coined by Gromov in $1987$. The
initials are in honor of Cartan, Alexandrov and Toponogov. For the
self-containedness of this paper, we will recall some basic
definitions here.

A metric $d$ of the metric space $(X,d\,)$ is called
\textit{intrinsic} if for every $P,Q \in X$\,$$
d\,(P,Q)\,=\,\inf_{\mathcal{L}}\,\{\,\Length\,(\mathcal{L})\,\}\,,$$
where the inf is taken over all rectifiable curves $\mathcal{L}$
joining the points $P$ and $Q$, and $\Length(\mathcal{L})$ is the
length of $\mathcal{L}$ measured in the metric $d$.

A curve $\mathcal{L}$ in a metric space $(X,d\,)$ joining a pair of
points $A,B$ is called a \textit{shortest arc} if its length is
equal to $d\,(A,B)\,.$

A metric space is said to be \textit{geodesically connected} or
\textit{a length space} if each pair of points in it can be joined
by a shortest arc.

An $\mathfrak{R}_K$ domain (also known as a $CAT(K)$ space) of the
metric space $(X,d\,)$ is a metric space with the following
properties:
\begin{enumerate}
\item[(i)] $\mathfrak{R}_K$ is a geodesically connected metric
space.
\item[(ii)] If $K>0$, then the perimeter of each triangle in
$\mathfrak{R}_K$ is less than $2\pi/\sqrt{K}$.
\item[(iii)] $K$-convexity: Each triangle $\triangle ABC\,\subset\,\mathfrak{R}_K$ and
its comparison triangle $\triangle
\overline{A}\overline{B}\overline{C}$ in the $K$-plane{\footnote{The
$K$-plane is the $2$-dimensional model space of constant Gaussian
curvature $K$, i.e., $\mathbf{R}^2$ if $K=0$, the standard Euclidean
$2$-hemisphere $\mathbf{S}_K$ of radius $1/\sqrt{K}$ if $K>0$ and the
hyperbolic plane of curvature $K$ if $K<0$\,. The comparison
triangle means $\triangle \overline{A}\overline{B}\overline{C}$ has
the same length of corresponding side as $\triangle ABC$\,,
measuring in respective metric. See \cite{BBI}. \label{64}}} have
the $CAT(K)$-inequality: $d\,(B,D)\,\leq\,
d_{K\text{-plane}}(\overline{B},\overline{D})$, where $\overline{D}$
is the point in the arc $\overline{A}\overline{C}$ such that
$d\,(A,D)\,=\, d_{K\text{-plane}}(\overline{A},\overline{D})$\,.
\end{enumerate}

A metric space $(X,d\,)$ is an Alexandrov space of curvature bounded
from above by $K$ if each point of $X$ is contained in some
neighborhood that is an $\mathfrak{R}_K$ domain. \vskip 2mm

To see the candidates for such Alexandrov spaces of curvature
bounded from above and relate the curvature in the sense of
Alexandrov and the sectional curvature of a Riemannian manifold, we
have the following theorem due to Alexandrov and Cartan.
 \begin{thm} \label{78} (\cite{A1}, \cite{Ca}) A smooth
Riemannian manifold $M$ is an Alexandrov space of curvature bounded
from above by $K$ if and only if the sectional curvature of $M$ is
$\leq K$\,.\end{thm}

\subsection{Energy of maps into metric spaces} \label{30} Let
$(\Sigma,g)$ be a $n$-dimensional compact Riemannian domain,
$d_{\Sigma}$ be the distance function on $\Sigma$ induced by $g$ and
$(X,d\,)$ be any complete metric space. A Borel measurable map $f:
\Sigma\to X$ is said to be in $L^2(\Sigma,X)$ if
$$
\int_{\Sigma}d^2(f(x),Q)d\mu<\infty
$$
for some $Q\in X.$ By the triangle inequality, this definition is
independent of the choice of $Q$.\vskip 2mm

 For $\epsilon>0$, let
$\Sigma_\epsilon=\{\eta\in\Sigma:
d_{\Sigma}(\eta,\partial\Sigma)>\epsilon\},
\,S_{\epsilon}(\eta)=\{\xi\in\Sigma:
d_{\Sigma}(\eta,\xi)=\epsilon\}$\,, $d\sigma_{\eta,\epsilon}(\xi)$
be the $(n-1)$-dimensional surface measure on $S_{\epsilon}(\eta)$
and $w_n$ be the area form of the unit sphere. For $u\in
L^2(\Sigma,X)$, construct an $\epsilon$-approximate energy function
$e_{\epsilon}: \Sigma\to \mathbf{R}$ by setting

\begin{equation}\label{83} e_{\epsilon}(\eta)=
  \left\{
   \aligned
   \frac{1}{w_n}\int_{S_{\epsilon}(\eta)}\frac{d^2(u(\eta),u(\xi))}{\epsilon^2}\frac{d\sigma_{\eta,\epsilon}(\xi)}{\epsilon^{n-1}}\quad \quad&{\text{for } \eta\in \Sigma_\epsilon, } \\
   0 \quad \quad&{\text{for } \eta\in \Sigma-\Sigma_\epsilon. } \\
   \endaligned
  \right.
\end{equation}

Define a linear functional $E_{\epsilon}: C_c(\Sigma)\to \mathbf{R}$
on the set of continuous functions with compact support in $\Sigma$
by setting
$$
E_{\epsilon}(f)=\int_{\Sigma} f e_{\epsilon} d\mu.
$$
\begin{defn} (\cite{KS}, 1.3ii) The map $u\in L^2(\Sigma,X)$ is said to
have finite energy or equivalently $u\in W^{1,2}(\Sigma,X)$ if
\begin{equation} \label{40}
E^{u}=\sup_{0\leq f\leq 1,f\in C_c(\Sigma)}\limsup_{\epsilon\to
0^{+}} E_{\epsilon}(f)< \infty.
\end{equation}
\end{defn}
The quantity $E^{u}$ is defined to be the energy of the map $u$. It
is shown in \cite{KS} that if $u$ has finite energy, then in fact
there exists a function $e(\eta)\in L^1(\Sigma)$ so that
$e_{\epsilon}(\eta)d\mu_g(\eta) \to e(\eta)d\mu_g(\eta)$ as
measures. The function $e(\eta)$ is called the energy density of $u$
and we write it as $|\nabla u|^2$ ( $|u'|^2$ when $n=1$) as an
analogue of the Riemannian case. In particular
$$
\Energy(u)=E^{u}=\int_{\Sigma}|\nabla
u|^2d\mu\,\,\,(\,=\int_{\Sigma}|u'|^2d\mu\quad \text{when } n=1).
$$
\begin{rem} \label{73} By definition, if $u,v\in W^{1,2}(\Sigma,X)$
then the pointwise distance function $d\,(u,v)\in
W^{1,2}(\Sigma,\mathbf{R})$ (see also theorem 1.12.2 of \cite{KS}). For closed curves $\alpha, \beta \in
W^{1,2}(\SS^1,X)$, the fact that $d\,(\alpha,\beta)\in
W^{1,2}(\SS^1, \mathbf{R})$ allows us the define the distance between
$\alpha$ and $\beta$ in $W^{1,2}(\SS^1,X)$ by setting
$$\dist(\alpha,\beta) \,=\,|d\, (\alpha , \beta)|_{W^{1,2}}\,.$$ Note that the Sobolev
embedding $C^0(\SS^1,\mathbf{R})\hookrightarrow
W^{1,2}(\SS^1,\mathbf{R})$ implies two $W^{1,2}$ curves that are
$W^{1,2}$ close are also $C^0$ close (cf.(\ref{44}))\,.
\end{rem}

\section  {Local energy convexity in $CAT(K)$ spaces} \label{38}
This section is devoted to prove Theorem \ref{24}. Equation (2.2iv) of \cite{KS} already gave the case of $K=0$ (with any $\rho>0$).
Our idea then follows from \cite{KS} to prove the case of $K>0$. We first provide a local distance convexity in the
standard Euclidean $2$-hemisphere $\mathbf{S}_K$ and then apply
Reshetnyak's majorization theorem to get the local energy convexity
for $W^{1,2}$ maps into a $CAT(K)$ space with $K\geq0$. In \cite {BN}, Berg and
Nikolaev defined the so-called $K$-quadrilateral cosine $cosq_K$ in
an Alexandrov space of curvature $\leq K$ which has the property
that $|cosq_K|\leq 1$. As we shall see in the following, this
quantity is much related to the local distance convexity in
$\mathbf{S}_K$.

\begin{lem} \label{17} (\cite {BN}) Consider a quadruple
$\mathcal{Q}=\{A,B,C,D\}$ of order points (see Figure \ref{18}), $A\neq B$ and $C\neq D$, in
$\mathbf{S}_K$ ($\mathbf{R}^2$ if $K=0$)\,. Let $k=\sqrt{K}$ and $d_{\mathbf{S}_K}$ be the
distance function in $\mathbf{S}_K$. Let
  $
   d_{\mathbf{S}_K}(A,B)=a, \,d_{\mathbf{S}_K}(C,D)=b, \,d_{\mathbf{S}_K}(A,D)=x,\, d_{\mathbf{S}_K}(B,C)=y, \,d_{\mathbf{S}_K}(A,C)=h \,\text{ and }
  \,d_{\mathbf{S}_K}(B,D)=i\,.
 $
Then the limit of the $K$-quadrilateral cosine equals to the $0$-quadrilateral cosine as $K\to 0$, i.e.,
\begin{align*} \label{35}
&  \lim_{K\to 0} cosq_K(\overrightarrow{DA},\overrightarrow{CB})\notag\\
=&\lim_{K\to 0}\frac{\cos ka+ \cos ky \cos kx \cos kb+\cos ka \cos kb-\cos ky
\cos kh-\cos kx \cos ki - \cos kh \cos ki}{(1+\cos kb)\sin kx \sin
ky}\notag\\
=&cosq_0(\overrightarrow{DA},\overrightarrow{CB})=\frac{a^2+b^2-h^2-i^2}{2xy} \,.
\end{align*}
\end{lem}

Based on Lemma \ref{17}, the following lemma follows directly from
an elementary computation (see Appendix \ref{36} for the detailed
computation)\,.

\begin{lem} \label{34} For any $x\in \mathbf{S}_K$, there exists
$\tau=\tau(K)>0$ such that if $\{A,D,C,B\}\subset B_{\tau}(x)\subset
\mathbf{S}_K$ is an ordered sequence and $E,F$ are the mid-points of
the shortest arcs $AB$ and $CD$ respectively, we have the following
distance convexity:
$$
\frac{1}{4}\left(d_{\mathbf{S}_K}(A,B)-d_{\mathbf{S}_K}(C,D)\right)^2\leq
d^2_{\mathbf{S}_K}(A,D)+d^2_{\mathbf{S}_K}(B,C)-2d^2_{\mathbf{S}_K}(E,F).
$$
\end{lem}

\noindent We will next recall Reshetnyak's majorization theorem in
1968 for an Alexandrov space of curvature bounded from above, which
is a far reaching generalization of the $K$-convexity that was
established by Alexandrov.

\begin{thm}  \label{23} (\cite{Re}) Let $(X,d\,)$ be an Alexandrov space
of curvature $\leq K$. In an $\mathfrak{R}_K$ domain of $X$, for
every rectifiable closed curve $\mathcal{L}$ with length less than
$2\pi/\sqrt{K}$ if $K>0$, there is a convex domain $\mathcal{V}$ in
the ${K\text{-plane}}$ and a map $\varphi: \mathcal{V} \to \mathfrak{R}_K$
such that $\varphi(\partial \mathcal{V})=\mathcal{L}$, the lengths
of the corresponding arcs coincide, and
$d_{K\text{-plane}}(\eta,\xi)\geq d\,(\varphi(\eta),\varphi(\xi))$,
for $\eta,\xi\in \mathcal{V}\,.$ \end{thm}

\begin{rem} \label{46} In particular, for an ordered sequence of
points $\{A,D,C,B\}$ in an Alexandrov space of curvature bounded from above by $K>0$, let $0\leq \lambda, \nu\leq 1$ be
given. Define $A_{\lambda}$ to be the point which is the fraction
$\lambda$ of the way from $A$ to $B$ (on the geodesic
$\gamma_{A,B}$). Let $D_{\nu}$ be the point which is the fraction
$\nu$ of the way from $D$ to $C$ (along the opposite geodesic
$\gamma_{D,C}$). By Theorem \ref{23}, there exists an ordered
sequence of points
$\{\overline{A},\overline{D},\overline{C},\overline{B} \} \subset
\mathbf{S}_K$ which are the consecutive vertices of a quadrilateral.
We can construct the corresponding points in $\mathbf{S}_K$ :
$$
\overline{A}_{\lambda}=(1-\lambda)\overline{A}+\lambda
\overline{B},\quad \overline{D}_\nu=(1-\nu)\overline{D}+\nu
\overline{C}.
$$
Then by Theorem \ref{23}
\begin{align}
&d\,(A,B)=d_{\mathbf{S}_K}({\overline{A}, \overline{B}}),\quad
d\,(C,D)=d_{\mathbf{S}_K}({\overline{C}, \overline{D}}),\notag\\
&d\,(A,D)=d_{\mathbf{S}_K}({\overline{A}, \overline{D}}),\quad
d\,(B,C)=d_{\mathbf{S}_K}({\overline{B}, \overline{C}}),\notag\\
&d\,({A}_{\lambda},{D}_\nu)\leq
d_{\mathbf{S}_K}\left({\overline{A}_{\lambda},\overline{D}_\nu}\right).\notag\end{align}
We call $\{\overline{A},\overline{D},\overline{C},\overline{B} \}$
the subembedding of $\{A,D,C,B\}$.\end{rem}

\begin{lem}  \label{22} Let $(X,d\,)$ be an Alexandrov space of curvature $\leq K$. In an $\mathfrak{R}_K$ domain of $x\in X$,
there exists $\,\rho=\rho(x,K)>0\,$ such that if $\{A,D,C,B\}\subset
B_{\rho}(x)\subset \mathfrak{R}_K$ is an ordered sequence and $E,F$
are the mid-points of the shortest arcs $AB$ and $CD$ respectively,
we have
\begin{equation}
\frac{1}{4}\left(d\,(A,B)-d\,(C,D)\right)^2\leq
d^2(A,D)+d^2(B,C)-2d^2(E,F).
\end{equation} \end{lem}

\begin{proof}  Equation (2.2iii) of \cite{KS} gives the case of $K=0$ (with any $\rho>0$)\,. For $K>0$, let $\varrho \,=\, \min\, \{d\,(x,y)\,|\, y\in\partial\, \mathfrak{R}_K\,\}$ and $\rho=\min\{\tau/4,
\varrho\}$ where $\tau$ is from Lemma \ref{34}. Let
$\{\overline{A},\overline{D},\overline{C},\overline{B}\}\subset
\mathbf{S}_K$ be a subembedding of $\{A,D,C,B\}$.  We see first of
all that $\overline{A},\overline{D},\overline{C}$ and $\overline{B}$
have to be in a geodesic ball of radius at most $4\rho\leq \tau$ in
$\mathbf{S}_K$ and thus satisfy the condition of Theorem \ref{23}.
Then by Theorem \ref{23} and Remark \ref{46}, letting
$\lambda=\nu=\frac{1}{2}$, we obtain
\begin{align} \label{71}
\frac{1}{4}(d\,(A,B)-d\,(C,D))^2 &
=\frac{1}{4}(d_{\mathbf{S}_K}({\overline{A},
\overline{B}})-d_{\mathbf{S}_K}({\overline{C},
\overline{D}}))^2\notag\\
 &\leq
d^2_{\mathbf{S}_K}({\overline{A},
\overline{D}})+d^2_{\mathbf{S}_K}({\overline{B},
\overline{C}})-2d^2_{\mathbf{S}_K}({\overline{A}_{\frac{1}{2}},
\overline{D}_{\frac{1}{2}}})\notag\\
 &\leq
d^2(A,D)+d^2(B,C)-2d^2(A_{\frac{1}{2}},D_{\frac{1}{2}})\,,
\end{align}
completing the proof\,.
\end{proof}

\begin{rem} For general $\lambda,\,\nu\,\in [0,1]$, a similar distance
convexity still holds with coefficients in terms of $\lambda$ and
$\nu$\,.\end{rem}

\begin{proof}(of {\bf{Theorem \ref{24}}})\quad For $u,v\in W^{1,2}(\Sigma, X)$
with images staying in $B_{\rho}(x)$, set
$\{A=u(\xi),B=v(\xi),C=v(\eta),D=u(\eta)\}$ as in Lemma \ref{22}\,,
we have:
\begin{equation} \label{26}
\frac{1}{4}(d\,(u(\eta),v(\eta))-d\,(u(\xi),v(\xi)))^2\leq
d^2(u(\xi),u(\eta))+d^2(v(\xi),v(\eta))-2d^2(w(\xi),w(\eta)),
\end{equation}
where $w(\xi)=\frac{u+v}{2}(\xi)$ is the mid-point of the geodesic
connecting $u(\xi)$ and $v(\xi)$\,.\vskip 1mm

Multiplying (\ref{26}) by $f(\eta)$ (where $0\leq f\leq 1$ and $f\in
C_c(\Sigma)$), averaging on the subset $\{|\eta-\xi|<\varepsilon\}$
of $\Sigma\times\Sigma$ and integrating over $\Sigma$ (as in 1.3 of
\cite{KS} and see \eqr{83}), then first of all we conclude that
$w\in W^{1,2}(\Sigma,X)$. By theorem 1.6.2 and theorem 1.12.2 of \cite{KS}
we obtain that for any $f\in C_c(\Sigma), 0\leq f\leq 1$:
$$
\frac{1}{4}\int_{\Sigma}f|\nabla d\,(u,v)|^2\leq
\int_{\Sigma}f|\nabla u|^2+\int_{\Sigma}f|\nabla
v|^2-2\int_{\Sigma}f|\nabla w|^2.
$$

Hence by definition (\ref{40}), we have an analogue to (2.2iv) of
\cite{KS}:
\begin{equation} \label{27}
\frac{1}{4}\int_{\Sigma}|\nabla d(u,v)|^2\leq E^{u}+E^{v}-2E^{w}.
\end{equation}
\end{proof}

 \begin{rem}   An immediate corollary of (\ref{27}) is
the uniqueness of the solution to the Dirichlet problem into a
$CAT(K)$ (with $K>0$) space with small image assumption, see \cite{Se} and cf.
theorem 2.2 of \cite{KS}.\end{rem}

\begin{cor} (\cite{Se}) Let $(\Sigma,g)$ be a Lipschitz Riemannian domain
and $(X,d\,)$ be an Alexandrov space of curvature bounded from above
by $K>0$. Fix a point $Q\in X$, Let $\phi\in W^{1,2}(\Sigma,X)$ with
$\phi(\Sigma)\subset B_{\rho}(Q)$ where $\rho=\rho(Q,K)$ is given by
Theorem \ref{24}. Define
$$
W^{1,2}_{\phi}=\{u\in W^{1,2}(\Sigma,X) \quad\big|\quad
u(\Sigma)\subset B_{\rho}(Q) \quad\text{and}\quad tr(u)=tr(\phi)\}.
$$
Then there exists a unique $u\in W^{1,2}_{\phi}$ which satisfies
$$
E^{u}=\int_{\Sigma}|\nabla u|^2d\mu=E_0=\inf_{v\in
W^{1,2}_{\phi}}E^{v}.
$$
\end{cor}

\section{Existence of good sweepouts by curves}\label{Goods}

Throughout the rest of this paper, we will let $(X,d\,)$ be a closed Alexandrov space of curvature bounded from above by some $K$.
Using the compactness of $X$, we let
\begin{equation}\label{66} \rho=\inf_{x\in
X}\,\{\,\rho(x,K)\,\}\,>0\,,\end{equation} where $\rho(x,K)$ is as
in Theorem \ref{24}. Fix a large positive integer $L$ and
let $\Lambda$ denote the space of piecewise linear maps (constant
speed geodesics) from $\SS^1$ to $X$ with exactly $L^2$ breaks
(possibly with unnecessary breaks) such that the length of each
geodesic segment is at most $\rho$ defined by (\ref{66}),
parametrized by a (constant) multiple of arclength and with
Lipschitz bound{\footnote{Note that a $W^{1,2}$ curve is also a
$C^{1/2}$ curve but not necessarily Lipschitz continuous, here the
Lipschitz bound denotes the bound of the speed and is equivalent to
the square of the $C^{1/2}$ bound of the curve. See \eqr{44}.}} $L$. Let $G \subset \Lambda$ denote the set of (possibly
self-intersecting) closed geodesics in $X$ of length at most $\rho
L^2$. (The constant speed of a curve in $\Lambda$ is equal to its
length divided by $2\pi$; and its energy is equal to its length
squared divided by $2\pi$. In other words, energy and length are
essentially equivalent, see (\ref{50}) and (\ref{44}))\,.\vskip 1mm

\subsection{The width} In \cite{CM1}, Colding and Minicozzi
introduced two crucial geometric concepts: sweepout and width. We
will recall and extend these definitions to a closed Alexandrov
space of curvature bounded from above.

\begin{defn}
\label{79} A continuous map $\sigma : \SS^1 \times [-1,1] \to X$ is
called a \textit{sweepout} in $X$, if for each $s$ the map $\sigma
(\cdot , s )$ is in $W^{1,2}(\SS^1,X)$, the map  $s \to \sigma
(\cdot , s )$ is continuous (in the induced metric as in Remark
\ref{73})
 from $[-1,1]$ to
$W^{1,2}(\SS^1,X)$, and finally $\sigma$ maps $\SS^1 \times \{ -1
\}$ and $\SS^1 \times \{ 1 \}$ to points.
\end{defn}
Let $\Omega$ be the set of sweepouts in $X$. Given a map
    $\hat{\sigma} \in \Omega$, the homotopy class $\Omega_{\hat{\sigma}}$
 is defined to be the set of maps $\sigma \in \Omega$ that are homotopic to
    $\hat{\sigma}$ through maps in $\Omega$.
\begin{defn}
\label{80} The \textit{width} $W = W (\hat{\sigma})$ associated to
the homotopy class
 $\Omega_{\hat{\sigma}}$ is defined by taking the infimum of the maximum of
the energy of each slice. That is,  set
\begin{equation}
    W = \inf_{  \sigma \in \Omega_{\hat{\sigma}}  } \,
       \,  \max_{ s \in [-1, 1]} \,  \Energy \, (\sigma (\cdot , s ))
          \, ,
\end{equation}
where the energy is the energy of maps into metric spaces given in
subsection \ref{30}, namely, \begin{equation} \label{50}
\text{Energy} (\sigma(\cdot,s))= \displaystyle\sup_{\substack {f\in
C_c(\SS^1)\\0\leq f\leq 1}}\limsup_{\epsilon\to
0^{+}}\int_{{\bf{S}}^1}f\left(\frac{d^2\left({\sigma} (\eta-\epsilon,
s), {\sigma} (\eta, s)\right)+d^2\left({\sigma} (\eta, s), {\sigma}
(\eta+\epsilon , s)\right)}
    {2\epsilon^2}\right) \, d\eta\,.
\end{equation}
\end{defn}
 We write $\Energy
(\sigma(\cdot,s))=E(\sigma(\cdot,s))=\int_{{\bf{S}}^1}|\sigma'(x,s)|^2dx$.
We shall see that a sweepout in $X$ induces a map from sphere
$\SS^2$ to $X$ and the width is always non-negative and is positive
if $\widetilde{\sigma}$ is in a non-trivial homotopy
class{\footnote{A particularly interesting example is when $X$ is a
topological $2$-sphere with $\pi_1(X)=\{0\}$ and the map induced by
a sweepout from $\SS^2$ to $X$ has degree one. In this case, the
width is positive and realized by a non-trivial closed geodesic with
index $1$\,, see footnote 2 of \cite{CM1}.\label{59}}}\,.
\begin{rem}
The $\epsilon$-approximate length function of $\sigma$ converges to a
$L^1$ function, which coincides with the speed function of $\sigma$,
as $\epsilon\to 0^{+}$, namely (see lemma 1.9.3 of \cite{KS}),
$$
\lim_{\epsilon\to 0^{+}}\frac{d\left(\sigma (\eta-\epsilon), \sigma(\eta)\right)+d\left(\sigma (\eta), \sigma (\eta+\epsilon)\right)}
    {2\epsilon}\,=\,|\sigma'|(\eta)\quad \text{a.e.}\,\,\eta\in \SS^1\,.
$$
\noindent Throughout the rest of this paper we will use $|\sigma'|$
to denote the speed function of a curve $\sigma$ in $X$.
\end{rem}

\subsection{Curve shortening $\Psi$} \label{39}

The curve shortening is a map $\Psi: \Lambda \to \Lambda$ so
that (see also section 2 of \cite{Cr})
\begin{itemize}
\item[(1)] $\Psi(\gamma)$ is
homotopic to $\gamma$ and $\Length (\Psi(\gamma)) \leq \Length
(\gamma)$.
\item[(2)] $\Psi (\gamma)$ depends continuously on $\gamma$.
\item[(3)] There is a continuous function $\phi:[0,\infty) \to
[0,\infty)$ with $\phi (0) = 0$ so that
\begin{equation}
    \dist^2 (\gamma , \Psi (\gamma)) \leq \phi \left( \frac{\Length^2 (\gamma) -
\Length^2 (\Psi (\gamma))}{\Length^2 (\Psi (\gamma))} \right) \, .
\end{equation}
\item[(4)] Given $\epsilon > 0$, there exists $\delta > 0$ so that
if $\gamma \in \Lambda$ with $\dist (\gamma , G) \geq \epsilon$,
then $\Length \, (\Psi (\gamma)) \leq \Length \, (\gamma) - \delta$.
\end{itemize}

We will use local linear replacement to define the curve shortening
map $\Psi$ which is identical to \cite{CM1}: fix a partition of
$\SS^1$ by choosing $2L^2$ consecutive evenly spaced
points{\footnote{Note that this is not necessarily where the
piecewise linear maps have breaks.}}
\begin{equation}
    x_0 , x_1 , x_2, \dots , x_{2L^2} = x_0 \in \SS^1\,,\quad \text{so that\,\,} |x_{j} - x_{j+1}| = \frac{\pi}{L^2}\leq\frac{\rho}{2L}\,.
\end{equation}
$\Psi (\gamma)$ is given in the following three steps:

\vskip1mm \noindent {\it{Step 1}}: Replace $\gamma$ on each
{\emph{even}} interval, i.e., $[x_{2j} , x_{2j+2}]$, by the linear
map with the same endpoints to get a piecewise linear curve
$\gamma_e: \SS^1 \to X$. Namely, for each $j$, we let $\gamma_e
\big|_{[x_{2j},x_{2j+2}]}$ be the unique shortest (constant speed)
geodesic from $\gamma(x_{2j})$ to $\gamma(x_{2j+2})$.

\vskip1mm \noindent {\it{Step 2}}: Replace $\gamma_e$ on each
{\emph{odd}} interval by the linear map with the same endpoints to
get the piecewise linear curve $\gamma_o: \SS^1 \to X$.

\vskip1mm \noindent {\it{Step 3}}: Reparametrize $\gamma_o$ (fixing
$\gamma_o (x_0)$) to get the desired  constant speed curve
$\Psi(\gamma) : \SS^1 \to X$.

It is easy to see that $\Psi$ maps $\Lambda$ to $\Lambda$ and has
property (1); cf. section 2 of \cite{Cr}.  The proof of properties
$(2), (3)$ and $(4)$ for $\Psi$ is virtually the same as \cite{CM1}.
We shall remark that there is a difficulty in the proofs of these properties: the second fundamental form
(smoothness) of the manifold is used in the
Riemannian case in \cite{CM1}, while we don't have the smoothness in an Alexandrov
space of curvature bounded above. But note that the local energy
convexity in Theorem \ref{24} requires only that the two curves both
stay in a small region, while the key lemma 4.2 in \cite{CM1}
requires that the two curves have the same endpoints.  This
fact allows us to get around this difficulty (see (\ref{76})). For
the completeness of this paper, we include the proofs of properties (3) and
(4) in Appendix \ref{51} and the proof of property $(2)$ in Appendix
\ref{75}. Throughout the rest of this section, we will assume these
properties of $\Psi$ and use them to prove the main theorem.

Combining  properties (3) and (4) of $\Psi$, we have the following
key lemma, which is crucial in producing the desired sequence of
good sweepouts.

\begin{lem} \label{41}
Given $W \geq 0$ and $\epsilon > 0$, there exists $\delta > 0$  so
that if $\gamma \in \Lambda$ and
\begin{equation}
   2\pi \, (W - \delta) <  \Length^2 \, ( \Psi (\gamma) ) \leq  \Length^2 \, ( \gamma)  < 2\pi \, (W + \delta ) \, ,
\end{equation}
then  $\dist ( \Psi (\gamma) , G) < \epsilon$.
\end{lem}

\begin{proof} If $W \leq \epsilon^2/6$,   then $\delta = \epsilon^2/6$ gives
 $\Length \, ( \Psi (\gamma) ) \,\leq\, 2 \epsilon$\,. This tells us the bound on distance of $\Psi (\gamma) $ to a point curve (e.g., its
 mid-point) which is a trivial closed geodesic in $G$.

Assume next that $W > \epsilon^2/6$.
 The triangle
inequality    gives
\begin{equation}
    \dist ( \Psi (\gamma) , G) \leq
        \dist ( \Psi (\gamma) , \gamma) + \dist (  \gamma , G)
     \, .
\end{equation}
Since $\Psi$  does not decrease  the length of $\gamma$ by much by
the assumption, property (4) of $\Psi$ bounds $\dist ( \gamma , G)$
by $\epsilon/2$ as long as $\delta$ is sufficiently small.
Similarly, property (3) of $\Psi$ allows us to bound $\dist ( \Psi
(\gamma) , \gamma)$ by $\epsilon/2$ as long as $\delta$ is
sufficiently small.
\end{proof}

\subsection{Defining the good sweepouts}\label{DefGood}

Choose a sequence of maps $\hat{\sigma}^j \in \Omega_{\hat{\sigma}}$
 with
\begin{equation} \label{33}
    \max_{s \in [-1,1]} \, \, \Energy \, (\hat{\sigma}^j (\cdot , s))
    < W + \frac{1}{j} \, .
\end{equation}
Observe that (\ref{33}) and the Cauchy-Schwarz inequality imply a
uniform bound for the length and uniform $C^{1/2}$ continuity for
the slices,   that are both independent of   $j$ and $s$. They
follow immediately from the following: for any small $\delta>0,\, [x,y]\subset
[0,2\pi]$ we pick $f\in C_c([0,2\pi]), 0\leq f\leq 1$, with $f=1$ on
$(x,y)$ and $supp\,(f)\subset [x-\delta, y+\delta]\subset [0,2\pi]$,
then
\begin{align} \label{44}
     & d^2\big( \hat{\sigma}^j (x , s) \,, \hat{\sigma}^j (y ,
    s)\big)\,\,\leq \, \Length^2\left(\hat{\sigma}^j(\cdot, s)|_{[x,y]}\right)\\
    =&{ \,\lim_{\delta\to 0^{+}}\limsup_{\epsilon\to 0^{+}}\,\left( \int_{x-\delta}^{y+\delta}
\,f\, \frac{d\left(\hat{\sigma}^j (\eta-\epsilon, s), \hat{\sigma}^j
(\eta, s)\right)+d\left(\hat{\sigma}^j (\eta, s), \hat{\sigma}^j
(\eta+\epsilon , s)\right)}{2\epsilon}\, d\eta \right)^2 \notag} \\
    \leq&{ \,|y-x| \, \lim_{\delta\to 0^{+}}\limsup_{\epsilon\to 0^{+}}  \,\int_{x-\delta}^{y+\delta}\,f^2\,\left(\frac{d^2\left(\hat{\sigma}^j
(\eta-\epsilon, s), \hat{\sigma}^j (\eta,
s)\right)+d^2\left(\hat{\sigma}^j (\eta, s), \hat{\sigma}^j
(\eta+\epsilon , s)\right)}{2\epsilon^2}\right) d\eta \notag}\\
     =&{ \,|y-x| \, \Energy\left(\hat{\sigma}^j(\cdot, s)|_{[x,y]}\right)\leq\,|y-x| \, (W+1)\,\notag}.
\end{align}

In order to get started and be able to use the properties of $\Psi$,
we would like all the initial curves to be in $\Lambda$.
 We will replace the $\hat{\sigma}^j$'s by sweepouts
$\sigma^j$ that, in addition to  satisfying (\ref{33}), also satisfy
that the slices $\sigma^j (\cdot , s)$ are in $\Lambda$.
  We will do this by using local
linear replacement similar to the construction of $\Psi$.  Namely,
the uniform $C^{1/2}$ bound for the slices allows us to fix a
partition of points $y_0 , \dots , y_N = y_0$ in $\SS^1$ so that
each interval $[y_i , y_{i+1}]$ is always mapped to a geodesic ball
in $X$ of radius at most  $\rho$. Next, for each $s$ and each $j$,
we replace $\hat{\sigma}^j (\cdot , s) \, \big|_{[y_{i},y_{i+1}]}$
by the linear map (geodesic) with the same endpoints  and call the
resulting map $\tilde{\sigma}^j (\cdot , s)$.  Reparametrize
$\tilde{\sigma}^j (\cdot , s)$ to have constant speed to get
$\sigma^j (\cdot , s)$.
 It is easy to see that each $\sigma^j (\cdot , s)$
satisfies (\ref{33}). Furthermore, the length bound for $\sigma^j
(\cdot , s)$ also gives a uniform Lipschitz (speed) bound for the linear
maps; let $L$ be this bound and $N\leq L^2$.

We can see from the proof of property $(2)$ for $\Psi$ in Appendix
\ref{75} that $\sigma^j$ is continuous in the transversal direction
(i.e. with respect to $s$) and homotopic to $ \hat{\sigma}$ in
$\Omega$, cf. \cite{B1}, \cite{B2}, section $2$ of \cite{Cr} and
appendix B of \cite{CM1}.\vskip 1mm

Finally, applying the  replacement map $\Psi$ to   each $\sigma^j
(\cdot , s)$ gives a new  sequence of sweepouts $ \gamma^j = \Psi
(\sigma^j)$.  ($\Psi$ depends continuously on $s$ and preserves the
homotopy class $\Omega_{\hat{\sigma}}$; it is clear that $\Psi$
 fixes the constant maps at $s = \pm 1$.)

\subsection{Almost maximal implies almost critical}
We will show that the sequence $\gamma^j= \Psi (\sigma^j)$ of
sweepouts is tight in the sense of the Introduction. Namely, we have
the following main theorem.

\begin{thm}     \label{mainThm}
Given $W \geq 0$ and $\epsilon > 0$, there exists $\delta > 0$   so
that if $j
> 1/\delta$ and for some $s_0$
\begin{equation}    \label{42}
    2\pi \, \Energy \, ( \gamma^j (\cdot , s_0))
        = \Length^2 \, ( \gamma^j (\cdot , s_0)) > 2\pi \, (W - \delta)  \, ,
\end{equation}
then for this $j$ we have $\dist \, \left(  \gamma^j (\cdot , s_0)
\, , \, G \right) < \epsilon$.
\end{thm}

 \begin{proof}
 Let $\delta$ be given by
Lemma \ref{41}.  By (\ref{42}), (\ref{33}), and using that $j
> 1/\delta$, we get
\begin{equation}
2\pi \, (W - \delta) <  \Length^2 \, ( \gamma^j (\cdot , s_0)) \leq
\Length^2 \, ( \sigma^j (\cdot , s_0)) <   2\pi \, (W+ \delta)  \, .
\end{equation}
Thus, since $\gamma^j (\cdot , s_0) = \Psi ( \sigma^j (\cdot ,
s_0))$, Lemma \ref{41}   gives $\dist ( \gamma^j (\cdot , s_0) \, ,
\, G) < \epsilon$.
\end{proof}

\section {Generalized Birkhoff-Lyusternik theorem}\label{GBL}

\subsection{Parameter spaces}\label{ParameterSpace} Instead of using the interval $[-1,1]$, as parameter space for the circles in the definition of sweepout (see Definition \ref{79}) and assuming that the curves start and end in point curves, one could have use any compact set $\mathcal{P}$ and require that the curves are constant on $\partial \mathcal{P}$ (or that $\partial \mathcal{P}=\emptyset$). Then we let $\Omega^{\mathcal{P}}$ be the set of continuous maps $\sigma: \SS^1\times \mathcal{P} \to X$ so that for each $s\in \mathcal{P}$ the curve $\sigma(\cdot, s)$ is in $W^{1,2}(\SS^1, X)$, the map $s\to \sigma(\cdot, s)$ is continuous from $\mathcal{P}$ to $W^{1,2}(\SS^1, X)$ and finally $\sigma$ maps $\partial \mathcal{P}$ to point curves. Given a map $\hat{\sigma}\in \Omega^{\mathcal{P}}$, the homotopy class $\Omega^{\mathcal{P}}_{\hat{\sigma}}\subset \Omega^{\mathcal{P}}$ is defined to be the set of maps $\sigma \in \Omega^{\mathcal{P}}$ that are homotopic to $\hat{\sigma}$ through maps in $\Omega^{\mathcal{P}}$. Finally the \textit{width} $W=W(\hat{\sigma})$ is
\begin{equation}\label{WidthPara}
 W\,=\, \inf_{\sigma\in \Omega^{\mathcal{P}}_{\hat{\sigma}}} \max_{s\in \mathcal{P}}\text{ Energy}(\sigma(\cdot,s))\,.
\end{equation}
Theorem \ref{mainThm} holds for these general parameter spaces and the proof is virtually the same.

\subsection{Generalized Birkhoff-Lyusternik theorem} The following is devoted to the proof of Theorem \ref{BLthm}.
\begin{proof} We will divide our proof into two cases. In the case of the fundamental group $\pi_1(X) \neq 0$, we can choose a non-contractible closed curve $\sigma_0: \SS^1\to X$. Then by the definition of width (see Definition \ref{80}) we see that $W>0$. It follows immediately from Theorem \ref{mainThm} that there exists at lease one non-trivial closed geodesic in $X$ if we apply the width-sweepout construction procedure as in section \ref{Goods}.

In the case of $\pi_1(X) = 0$, i.e., $X$ is simply-connected (or $1$-connected), then it's well known that $H_1(X)\cong \pi_1(X)/[\pi_1(X),\pi_1(X)]$ (see e.g. \cite[Theorem 2A.1]{Hat}, page 166) and thus $H_1(X)=0$. Then by the assumption of the theorem, there exists the first nonzero $k_1$-th homology group $H_{k_1}(X)\neq 0$ for some integer $k_1$ with $2\leq k_1\leq k_0$. Therefore by the Hurewicz theorem which states that the first nonzero homotopy and homology groups of a simply-connected space occur in the same dimension and are isomorphic (see e.g. \cite[Theorem 4.32]{Hat}, page 366), we have
$$\pi_{k_1}(X) \cong H_{k_1}(X)\neq 0\,.$$
Thus there is a non-contractible map $$\omega_0 : \SS^{k_1} \to X$$
from the $k_1$-sphere $\SS^{k_1}$ to $X$ for $k_1\geq 2$\,. Note that $\SS^{k_1}$ is equivalent to $\SS^1\times \bar{B}^{k_1-1}/ \sim$, where $\sim$ is the equivalence relation $(\theta_1, y)\sim (\theta_2, y)$ where $\theta_1, \theta_2 \in \SS^1$ and $y\in \partial \bar{B}^{k_1-1}$. Here $\bar{B}^{k_1-1}$ is the closed unit ball in $\mathbf{R}^{k_1-1}$\,. We use this decomposition of $\SS^{k_1}$ to define the width of $X$.

Take $\mathcal{P}= \bar{B}^{k_1-1}$ as the parameter space as in subsection \ref{ParameterSpace} and define the width $W$ as in \eqr{WidthPara}.  We see directly from the fact that $\omega_0$ is non-contractible that $W>0$. Again, it follows immediately from Theorem \ref{mainThm} that there exists at lease one non-trivial closed geodesic in $X$ if we apply the width-sweepout construction procedure as in section \ref{Goods}.
\end{proof}

\appendix

\section{Establishing properties (3) and (4) of $\Psi$} \label{51}

To prove property (3) of $\Psi$, we will use the following equivalent way to construct $\Psi (\gamma)$ :
\begin{enumerate}
\item[($A_1$)]  Follow Step 1 to get $\gamma_e$.
\item[($B_1$)] Reparametrize $\gamma_e$ (fixing the image of $x_0$) to get the
constant speed curve $\tilde{\gamma}_e$.  This reparametrization
moves the points $x_j$ to new points $\tilde{x}_j$ (i.e., $\gamma_e
(x_j) = \tilde{\gamma}_e (\tilde{x}_j)$). \item[($A_2$)] Do linear
replacement on the odd $\tilde{x}_j$ intervals to get
$\tilde{\gamma}_o$. \item[($B_2$)] Reparametrize $\tilde{\gamma}_o$
(fixing the image of $x_0$) to get the constant speed curve $\Psi
(\gamma)$.
\end{enumerate}
One sees easily that this gives the same curve since $\tilde{\gamma}_o$
is just a reparametrization of  ${\gamma}_o$. We also see that
each of the four steps is energy non-increasing{\footnote{This is obvious for
the linear replacements, since linear maps minimize energy. It
follows  from (\ref{44}) for the reparametrizations, since for a
curve $\sigma:\SS^1 \to X$ we have $\Length^2 (\sigma) \leq 2\pi \, \Energy (\sigma) \, ,$
with equality if and only if its speed is a constant $=\Length
(\sigma)/(2\pi)$ almost everywhere.}}. Thus property (3) follows from the triangle inequality once
we bound $\dist (\gamma , \gamma_e)$ and $\dist (\gamma_e ,
\tilde{\gamma}_e)$ in terms of the decrease in length  (as well as
the analogs for steps $(A_2)$ and $(B_2)$). \vskip 1.5mm

The bound on $\dist (\gamma , \gamma_e)$  follows directly from the
next corollary of Theorem \ref{24}\,.

\begin{cor} \label{47} There exists $C\,$ so that if $I$ is an interval of
length at most $\rho/L$, $\sigma_1: I\to X$ is a curve with
Lipschitz bound $L$, and $\sigma_2:I\to X$ is the minimizing
geodesic with the same endpoints, then
$$
\dist^2(\sigma_1,\sigma_2)\leq C \left(E^
{\sigma_1}-E^{\sigma_2}\right).
$$
\end{cor}

\begin{proof} Let $\Sigma=I\subset \SS^1$ and note that $w=\frac{\sigma_1+\sigma_2}{2}$ has the same end points as
$\sigma_1$ and $\sigma_2$. Since $d\,(\sigma_1,\sigma_2)\in
W^{1,2}_0(I, \mathbf{R})$ (see theorem 1.12.2 of \cite{KS}) and from Theorem \ref{24}, the
Poincar$\acute{\text{e}}$ inequality and (\ref{27}) imply$$
\dist^2(\sigma_1,\sigma_2 )\,\leq\, C(I)\int_{I}|\nabla
d\,(\sigma_1,\sigma_2)|^2 d\mu \,\leq\,C \left(E^
{\sigma_1}-E^{\sigma_2}\right)\,,
$$
where we used the minimality of $\sigma_2$\,.
\end{proof}

Applying Corollary \ref{47} on each of $L^2$ intervals in step
$(A_1)$, we get that
\begin{equation}
\dist^2(\gamma, \gamma_e)\,\leq\,C
\left(E^{\gamma}-E^{\gamma_e}\right)\,\leq\,\frac{C}{2\pi}\left(\Length^2(\gamma)-\Length^2(\Psi(\gamma))\right)\,.
\end{equation}
\noindent This gives the desired bound on $\dist(\gamma, \gamma_e)$
since $\Length(\Psi(\gamma)) \,\leq\, \rho L^2$.\vskip 1.5mm

 To bound $\dist (\gamma_e , \tilde{\gamma}_e)$, we will use that
  $\gamma_e$ is just the
composition $\tilde{\gamma}_e \circ P$, where $P: \SS^1 \to \SS^1$
is a monotone piecewise linear map{\footnote{The map $P$ is
Lipschitz, but the inverse map $P^{-1}$ may not be if $\gamma_e$ is
constant on an interval.}} and let $L$ be its Lipschitz bound as
well. Using that the (piecewise constant) speed of $\gamma_e$ is
$|\gamma_e'|=|(\tilde{\gamma}_e \circ P)'| = |\tilde{\gamma}_e'\circ
P|\cdot |P'|\leq L^2$  (Note: $|\tilde{\gamma}_e'\circ
P|(x)$ denotes the speed of $\tilde{\gamma}_e$ at point $P(x)$) and the (constant) speed of
$\tilde{\gamma}_e= \,|\tilde{\gamma}_e'| = \Length
(\tilde{\gamma}_e)/(2\pi) \leq L$ (away from the breaks), and also that
the integral of $P'$ is $2\pi$, we have
\begin{align}   \label{55}
    \int_{\SS^1} \left( P' - 1 \right)^2 &= \int_{\SS^1}   (P')^2 - 2\pi =
\int_{\SS^1}  \left( \frac{|\gamma_e'|}{|\tilde{\gamma}_e' \circ P|}
\right)^2 - 2\pi    =
      \frac{4\pi^2 }{\Length^2
(\tilde{\gamma}_e)}    \, \int_{\SS^1}  |\gamma_e'|^2 - 2\pi \notag
\\ &= 2\pi \,   \frac{ \Energy
    (\gamma_e) - \Energy (\tilde{\gamma}_e)}{\Energy
    (\tilde{\gamma}_e)}    \leq
2\pi \,   \frac{ \Energy
    (\gamma) - \Energy (\Psi (\gamma))}{\Energy
    (\Psi (\gamma))}
    \, .
\end{align}

Now divide $\SS^1$ into two sets, $S_1$ and $S_2$, where $S_1$ is
the set of points within distance $(\pi \, \int_{\SS^1}  |P'-1|^2
)^{1/2}$ of a break point for $\tilde{\gamma}_e$. Since $P(x_0) =
x_0$, we have $|P(x) - x| \leq (\pi \, \int_{\SS^1}  |P'-1|^2
)^{1/2}$. Since $\gamma_e$ and $\tilde{\gamma}_e$ agree at $x_0 =
x_{2L^2}$, the Wirtinger inequality{\footnote{The
 Wirtinger inequality is just the usual Poincar$\acute{\text{e}}$
inequality which bounds the $L^2$ norm in terms of the $L^2$ norm of
the derivative; i.e., $\int_{0}^{2\pi}f^2dt\leq 4\,
\int_{0}^{2\pi}(f')^2dt$ provided $f(0)=f(2\pi)=0$. \label{53} }}
bounds $\dist^2 (\gamma_e , \tilde{\gamma}_e)$ in terms of
\begin{equation} \label{54}
    \int_{\SS^1} \, \left| \nabla d\,(\tilde{\gamma}_e \circ P, \,\tilde{\gamma}_e) \right|^2 \leq
     \int_{S_1} \,
\big( | \,(\tilde{\gamma}_e \circ P)'| + | \, \tilde{\gamma}'_e\, |
\big)^2 + \int_{S_2} \, \left| \nabla d\,(\tilde{\gamma}_e \circ P,
\,\tilde{\gamma}_e) \right|^2\,,
\end{equation}
where we used that the fact (\ref{29}) in the proof of Lemma
\ref{34} (see last part of Appendix \ref{36}) implies
\begin{equation}\label{77} \int_{S_1} \, \left| \nabla
d\,(\tilde{\gamma}_e \circ P, \,\tilde{\gamma}_e) \right|^2\,\leq\,
\int_{S_1} \,\left(| \,(\tilde{\gamma}_e \circ P)'| + | \,
\tilde{\gamma}'_e\, |\right)^2\,.\end{equation}

We will bound both terms on the right hand side of (\ref{54}) in
terms of $\int_{\SS^1} |P'-1|^2$ and then appeal to (\ref{55})\,. To
bound the first term, we have
\begin{equation}
    \int_{S_1} \,
\big( | \,(\tilde{\gamma}_e \circ P)'| + | \, \tilde{\gamma}'_e\, |
\big)^2 \leq \, (L^2+L)^2 \, \Length (S_1)  \leq  8 \, L^{6} \,
\left(\pi \, \int_{\SS^1}  |P'-1|^2 \right)^{1/2} \, .
\end{equation}
We see that if $(\pi \, \int_{\SS^1}  |P'-1|^2
)^{1/2}\,\geq\,\frac{\pi}{2L^2}\,=\,\frac{|x_{j} - x_{j+1}|}{2}$, we
are done since in this case $S_2\,=\,\emptyset$\,.\vskip 2mm

On the other hand, suppose $(\pi \, \int_{\SS^1} |P'-1|^2
)^{1/2}\,<\,\frac{\pi}{2 L^2}$\,; \,note that if $x \in S_2$, then
$\tilde{\gamma}_e (x)$ and $\tilde{\gamma}_e \circ P (x)$ stay
within the $\rho$-neighborhood between two break points (although
$\tilde{\gamma}_e$ and $\tilde{\gamma}_e \circ P$ might not have the
same endpoints) and $|\tilde{\gamma}_e'\circ
P|=|\tilde{\gamma}_e'|\leq L$ in this neighborhood. Thus, we can
bound the second term by applying Theorem \ref{24}. Namely, by
summing up the integral over each piece of $S_2$, we have
\begin{align} \label{76}
&\frac{1}{4}\int_{S_2} \, \left| \nabla d\,(\tilde{\gamma}_e \circ P,
\,\tilde{\gamma}_e) \right|^2  \notag\,\\
&\leq \,
\Energy\left((\tilde{\gamma}_e \circ
P)|_{S_2}\right)+\Energy\left(\tilde{\gamma}_e|_{S_2}\right)-2\Energy\left(\left(\frac{\tilde{\gamma}_e
\circ
P+\tilde{\gamma}_e}{2}\right){\big|}_{S_2}\right)\notag\\
\,&=\, \int_{S_2} |(\tilde{\gamma}_e'\circ P)P'|^2+\int_{S_2}
|\tilde{\gamma}_e'|^2-2\int_{S_2}
\left(\frac{|(\tilde{\gamma}_e'\circ P)P'| + |\tilde{\gamma}_e'|}{2}\right)^2\notag\\
\,&= \,\int_{S_2} \frac{(|\tilde{\gamma}_e'\circ P|\cdot|P'| -
|\tilde{\gamma}_e'|)^2}{2} \leq\, \frac{L^2}{2} \int_{\SS^1}
|P'-1|^2 \,,
\end{align}
completing the proof of property (3).

To prove property (4) of $\Psi$,  suppose it is not true, namely, there exist $\epsilon
> 0$ and a sequence $\gamma_j \in \Lambda$ with  $\Energy (\Psi
(\gamma_j)) \geq \Energy (\gamma_j) - 1/j$ and $\dist (\gamma_j , G)
\geq \epsilon
> 0$; note that the second condition implies a positive lower
bound for $\Energy (\gamma_j)$.   Observe next that the space
$\Lambda$ is compact{\footnote{Compactness of $\Lambda$ follows
 since  $\sigma \in \Lambda$ depends continuously on
the images of the $L^2$ break points in the compact metric space
$X$.}} and, thus, a
 subsequence of the $\gamma_j$'s must converge to some $\gamma \in
 \Lambda$.  Since property (3)
implies that $\dist (\gamma_j , \Psi (\gamma_j)) \to 0$, the $\Psi
(\gamma_j)$'s also converge to $\gamma$.  The continuity of $\Psi$,
i.e., property (2) of $\Psi$, then implies that $\Psi (\gamma) =
\gamma$.  However, this implies that $\gamma \in G$ since the only
fixed points of $\Psi$ are (possibly self-intersected) closed
geodesics. This last fact follows immediately from Corollary
\ref{47} and (\ref{55}). However, this would contradict that the
$\gamma_j$'s remain a fixed distance from any such closed geodesic,
completing the proof of (4).

\section{The continuity of $\Psi$ } \label{75}

\begin{lem} Let $\gamma:\SS^1\to (X,d\,)$ be a $W^{1,2}$ map with
$\Energy(\gamma)\leq \rho L$. If $\gamma_e$ and $\tilde{\gamma}_e$
are given by applying $(A_1)$ and $(B_1)$ to $\gamma$, then the map
$\gamma\to \tilde{\gamma}_e$ is continuous from $W^{1,2}$ to
$\Lambda$ equipped with the $W^{1,2}$ norm as in Remark
\ref{73}.\end{lem}

\begin{proof} It follows from (\ref{44}) and the energy bound that $d\,\left(\gamma(x_{2j}),\gamma(x_{2j+2})\right)\leq
\rho$ for each $j$, and thus we can apply step $(A_1)$. Now suppose
that $\gamma^1$ and $\gamma^2$ are non-constant curves in $\Lambda$
(continuity at the constant maps is obvious). For $i=1,2$ and
$j=0,1,2,...,L^2-1$, let
$a^i_j=d\,\left(\gamma^i(x_{2j}),\gamma^i(x_{2j+2})\right)$. Let
$S^i=\frac{1}{2\pi}\sum_{j=0}^{L^2-1}a^i_j$ be the speed of
$\tilde{\gamma}^i_e$, so that $|(\tilde{\gamma}^i_e)'|=S^i$ except
at the $L^2$ break points. Since, by Remark \ref{73}, $W^{1,2}$
close curves are also $C^0$ close, it follows that the points
$\gamma_e(x_{2j})=\gamma(x_{2j})$ (identity map) are continuous with respect to the
$W^{1,2}$ norm. Thus the $a^i_j$'s are continuous functions of
$\gamma^i$, and so is each $S^i$. Moreover, the local energy
convexity in Theorem \ref{24} implies that the $\gamma_e^i$'s
are indeed $W^{1,2}$ close on each interval $[x_{2j},x_{2j+2}]$ if
the $\gamma^i$'s are (since $\gamma_e^i\big|_{[x_{2j},x_{2j+2}]}$'s also stay within a $\rho$-neighborhood
and the right-hand side of the energy convexity for them
is just a continuous function of $S^i$'s). Thus, we have shown $\gamma \to \gamma_e$ is
continuous.

To show $\gamma_e \to \tilde{\gamma}_e$ is also continuous, it
suffices to show that the $\tilde{\gamma}_e^i$'s are close when the
$\gamma^i_e$'s are. Since the point $x_0=x_{2L^2}$ is fixed under
the reparametrization, this will follow from applying Wirtinger's
inequality to
$d\,(\tilde{\gamma}_e^1,\tilde{\gamma}_e^2)-d\,(\tilde{\gamma}_e^1(x_0),\tilde{\gamma}_e^2(x_0))$
once we show that $\int_{\SS^1}|\nabla
d\,(\tilde{\gamma}_e^1,\tilde{\gamma}_e^2)|^2$ can be made small.

The piecewise linear curve $\tilde{\gamma}^i_e$ is linear on the
intervals
\begin{equation} \label{81}
    I^i_j = \left[ \frac{1}{S^i} \, \sum_{\ell < j} a^i_{\ell} \, , \,
        \frac{1}{S^i} \, \sum_{\ell \leq j} a^i_{\ell} \right] \, .
   \end{equation}
   Set $I_j = I^1_j \cap I^2_j$.
   Observe first that
    since the intervals $I^i_j$ in (\ref{81}) depend
   continuously on  $\gamma_e^i$,  the measure of the complement
   $\SS^1 \setminus \left[ \cup_{j=0}^{L^2-1} I_j \right]$ can be made small, so that
\begin{equation} \label{82}
\int_{\SS^1 \setminus \left[ \cup I_j \right]} \, \, \left|\nabla
d\,(\tilde{\gamma}_e^1,\tilde{\gamma}_e^2)\right|^2 \leq \int_{\SS^1
\setminus \left[ \cup I_j \right]} \, \,
\left(|(\tilde{\gamma}_e^1)'| +|(\tilde{\gamma}_e^2)'|
    \right)^2 \leq  4\,  L^2 \, \Length \, \left( \SS^1 \setminus \left[ \cup I_j
    \right] \right)
\end{equation}
can also be made small. We will divide the $I_j$'s into two groups,
depending on the size of $a^1_j$.
  Fix some $\epsilon > 0$ and suppose first that $a^1_j < \epsilon$; by continuity, we can assume that
  $a^2_j < 2\epsilon$.  For such $j$, we get
  \begin{equation}
    \int_{I_j}  \left|\nabla
d\,(\tilde{\gamma}_e^1,\tilde{\gamma}_e^2)\right|^2 \leq 2 \,
\int_{I_j^1}  \left| (\tilde{\gamma}_e^1)' \right|^2 + 2
\int_{I_j^2}
    \left| (  \tilde{\gamma}_e^2)' \right|^2 \leq 2 \, L \, \left( a^1_j + a^2_j \right) \leq 6 \, \epsilon \, L \, .
      \end{equation}
      Since there are at most $L^2$ breaks, summing over these intervals
      contributes at most $6\epsilon \, L^3$.

      On the other hand, suppose now $a^1_j \geq \epsilon$; by continuity we can assume that
  $a^2_j \geq \epsilon/2$.   In this case, $\tilde{\gamma}_e^i$
      can be written on $I_j$ as the composition
    $\gamma_e^i \circ
   P^i_j$ where $\left|
   (P^i_j)'\right| = 2 \pi \, S^i/ (L^2 a^i_j)$.  Furthermore, $P^1_j$ and $P^2_j$ both map $I_j$ into $[x_{2j}, x_{2j+2}]$
   and arguing as (\ref{76}) we have
   $$
   \frac{1}{4} \int_{I_j} \left|\nabla
d\,(\tilde{\gamma}_e^1,\tilde{\gamma}_e^2)\right|^2 =  \frac{1}{4}
 \int_{I_j} \left|\nabla d\,({\gamma}_e^1 \circ P_j^1 , {\gamma}_e^2
\circ
    P_j^2)\right|^2 \, \leq\, \frac{1}{2}\,\int_{I_j}
\left(|({\gamma}_e^1)'|\cdot|(P_j^1)'|-|({\gamma}_e^2)'|\cdot|(P_j^2)'|\right)^2.
    $$
This can be made small since the speed $\left|
   (P^i_j)'\right|$ is continuous in $\gamma^i$ and the (piecewise constant) speeds
   $|(\gamma_e^i)'|$'s are close when $\gamma_e^i$'s are.
   Therefore, the integral over these intervals can also be made
   small since there are at most $L^2$ of them.
\end{proof}

\section{Proof of Lemma \ref{34}} \label{36}
If $(X,d\,)$ has curvature bounded from above by $K>0$ in the sense
of Alexandrov, then $(X,\frac{\sqrt{K}d}{\sqrt{\epsilon}})$ has curvature bounded from
above by $\epsilon$, so that the local distance convexity in Lemma \ref{34} is homogenous w.r.t. $K$\,.
Hence, it suffices to assume the metric
space has curvature bounded from above by $\epsilon$ which is sufficiently small.\vskip 2mm

Suppose now $K>0$ is sufficiently small. Let  $d_{\mathbf{S}_K}(A,B)=a,\, d_{\mathbf{S}_K}(C,D)=b,\,
d_{\mathbf{S}_K}(A,D)=x, \,d_{\mathbf{S}_K}(B,C)=y,\,
d_{\mathbf{S}_K}({E,F})=g, \,d_{\mathbf{S}_K}({E,D})=c,\,
d_{\mathbf{S}_K}(A,F)=d, \,d_{\mathbf{S}_K}({B,F})=e$ and
$d_{\mathbf{S}_K}({E,C})=f$ (see Figure \ref{18}). In the rest of
this section we aim to prove the following that gives Lemma
\ref{34}: for $a,b,x,y$ small enough (i.e., under small region
assumption), we have the following inequality:
\begin{equation}
\frac{1}{4}(a-b)^2\leq x^2+y^2-2g^2\,.
\end{equation}

Based on Lemma \ref{17}, we first provide three key observations.
\input{figure02.TpX}
\begin{lem} \label{28} For $x,y$ sufficiently small and some $\alpha, \beta \in \mathbf{R}$ , we have:
\begin{align}
W(x^2+y^2-2g^2)=Sx^2+Ty^2+U+V-Rg^2
+O(x^2g^2)+O(y^2g^2)+O(g^4+x^4+y^4)\notag,
\end{align}
where
\begin{enumerate}
\item[(1)] \,$
W=\frac{k^4}{6}(2+\cos \frac{k}{2}a+\cos
\frac{k}{2}b)\left(c^2+d^2+e^2+f^2-\frac{1}{4}(a^2+b^2)-2\alpha xg-2\beta yg\right)\\
+2k^2(\cos\frac{k}{2}a+\cos \frac{k}{2}b)
-\frac{k^2}{2}(\cos kc+\cos kd+\cos ke+\cos kf)\,,$\\
\item[(2)] \,$R=\frac{k^4}{2}(2+\cos \frac{k}{2}a+\cos
\frac{k}{2}b)\left(c^2+d^2+e^2+f^2-\frac{1}{3}(a^2+b^2)-2\alpha xg-2\beta yg\right)\\
+6k^2(\cos\frac{k}{2}a+\cos \frac{k}{2}b)-2k^2(\cos kc+\cos kd+\cos ke+\cos kf)\,,$\\
\item[(3)]\,$S=\frac{k^4}{6}(2+\cos \frac{k}{2}a+\cos
\frac{k}{2}b)(e^2+f^2-2\beta yg)+k^2(\cos \frac{k}{2}a+\cos
\frac{k}{2}b)+\frac{k^2}{2}(\cos kc+\cos kd-\cos ke-\cos kf)\,,$\\
\item[(4)]\,$T=\frac{k^4}{6}(2+\cos \frac{k}{2}a+\cos
\frac{k}{2}b)(c^2+d^2-2\alpha xg)+k^2(\cos \frac{k}{2}a+\cos \frac{k}{2}b)
-\frac{k^2}{2}(\cos kc+\cos kd-\cos ke-\cos kf)\,,$\\
\item[(5)]\,$ U=k^2(2+\cos\frac{k}{2}a+\cos
\frac{k}{2}b)\left(c^2+d^2+e^2+f^2-\frac{1}{2}(a^2+b^2)-2\alpha xg- 2\beta yg\right)\,, $\\
\item[(6)]\,$ V=8(\cos \frac{k}{2}a+1)(\cos \frac{k}{2}b+1)-4(\cos kc+1)(\cos
kd+1)-4(\cos ke+1)(\cos kf+1)\,.$
\end{enumerate}
\end{lem}

\begin{proof} Apply Lemma \ref{17} to $\{A,E,F,D\}$, we can choose $K$ sufficiently small to be determined later so that for some $\alpha$ (with $k=\sqrt{K}, |\alpha|$ sufficiently small)
\begin{align}
&{\alpha+ (\frac{1}{4}a^2+\frac{1}{4}b^2-c^2-d^2)\big/(2x g)=cosq_K(\overrightarrow{AD},\overrightarrow{EF})=cosq_K(\overrightarrow{DA},\overrightarrow{FE})\label{58}}\\
=&{\frac{\cos \frac{k}{2}a+ \cos kg \cos kx \cos \frac{k}{2}b+\cos
\frac{k}{2}a \cos \frac{k}{2}b-\cos kg \cos kd-\cos kx \cos kc - \cos kc
\cos kd}{(1+\cos \frac{k}{2}b)\sin kx \sin kg} \label{56}}\\
=&{\frac{\cos \frac{k}{2}b+ \cos kg \cos kx \cos \frac{k}{2}a+\cos
\frac{k}{2}a \cos \frac{k}{2}b-\cos kx \cos kd-\cos kg \cos kc - \cos kc
\cos kd}{(1+\cos \frac{k}{2}a)\sin kx \sin kg}\label{57}\,.}
\end{align}
\noindent By Taylor series expansions for sine and cosine (in $x$
and $g$) and using (\ref{58})-(\ref{56}), we have
\begin{align*}
& \left[\frac{1}{4}(a^2+b^2)-c^2-d^2+2\alpha xg\right](1+\cos
\frac{k}{2}b)(kx-\frac{1}{6}(kx)^3+O(x^5))(kg-\frac{1}{6}(kg)^3+O(g^5))\\
=&{2xg\big[(\cos \frac{k}{2}a)(1+\cos \frac{k}{2}b)+(\cos
\frac{k}{2}b)(1-\frac{1}{2}(kx)^2+\frac{1}{24}(kx)^4+O(x^6))(1-\frac{1}{2}(kg)^2+\frac{1}{24}(kg)^4}\\
&{+O(g^6))-(\cos kd)(1-\frac{1}{2}(kg)^2+\frac{1}{24}(kg)^4+O(g^6))-(\cos
kc)(1-\frac{1}{2}(kx)^2+\frac{1}{24}(kx)^4}\\
&{+O(x^6))-\cos kc\cos kd\big]\,.}
\end{align*}
 \noindent Combining the terms in $x^2$ and $g^2$
 yields
\begin{align*}
 &{\left[\frac{k^4}{6}(1+\cos
\frac{k}{2}b)(c^2+d^2-\frac{1}{4}(a^2+b^2)-2\alpha xg)-k^2\cos kd+k^2\cos\frac{k}{2}b\right] g^2}\\
=&{-\left[\frac{k^4}{6}(1+\cos
\frac{k}{2}b)(c^2+d^2-\frac{1}{4}(a^2+b^2)-2\alpha xg)-k^2\cos kc+k^2\cos
\frac{k}{2}b\right]x^2}\\
&{+k^2(1+\cos
\frac{k}{2}b)(c^2+d^2-\frac{1}{4}(a^2+b^2)-2\alpha xg)
+2(1+\cos \frac{k}{2}a)(1+\cos \frac{k}{2}b)}\\
&{-2(1+\cos kc)(1+\cos
kd)+O(x^2g^2)+O(x^4+g^4)\,.}
\end{align*}
\noindent Similarly, using (\ref{58})-(\ref{57}), we have
\begin{align*}
 &{\left[\frac{k^4}{6}(1+\cos
\frac{k}{2}a)(c^2+d^2-\frac{1}{4}(a^2+b^2)-2\alpha xg)-k^2\cos kc+k^2\cos\frac{k}{2}a\right] g^2}\\
=&{-\left[\frac{k^4}{6}(1+\cos
\frac{k}{2}a)(c^2+d^2-\frac{1}{4}(a^2+b^2)-2\alpha xg)-k^2\cos kd+k^2\cos
\frac{k}{2}a\right]x^2}\\
&{+k^2(1+\cos
\frac{k}{2}a)(c^2+d^2-\frac{1}{4}(a^2+b^2)-2\alpha xg)
+2(1+\cos \frac{k}{2}a)(1+\cos \frac{k}{2}b)}\\
&{-2(1+\cos kc)(1+\cos
kd)+O(x^2g^2)+O(x^4+g^4)\,.}
\end{align*}
\noindent Therefore,
 \begin{align*}\label{19}
 &{\big[\frac{k^4}{6}(2+\cos
\frac{k}{2}a+\cos \frac{k}{2}b)(c^2+d^2-\frac{1}{4}(a^2+b^2)-2\alpha xg)+k^2(\cos
\frac{k}{2}a+\cos \frac{k}{2}b}\\
&{-\cos kc-\cos kd)\big] g^2}\\
=&{-\big[\frac{k^4}{6}(2+\cos \frac{k}{2}a+\cos
\frac{k}{2}b)(c^2+d^2-\frac{1}{4}(a^2+b^2)-2\alpha xg)+k^2(\cos \frac{k}{2}a+\cos
\frac{k}{2}b}\\
&{-\cos kc-\cos kd)\big]x^2+k^2(2+\cos \frac{k}{2}a+\cos
\frac{k}{2}b)(c^2+d^2-\frac{1}{4}(a^2+b^2)-2\alpha xg)}\\
&{+4(1+\cos
\frac{k}{2}a)(1+\cos \frac{k}{2}b)-4(1+\cos kc)(1+\cos kd)+O(x^2g^2)+O(x^4+g^4)\,.}
\end{align*}

\noindent Similarly, in $\{B,E,F,C\}$ we have for some $\beta$ (with $|\beta|$ sufficiently small)

\begin{align*}
 &{\big[\frac{k^4}{6}(2+\cos
\frac{k}{2}a+\cos \frac{k}{2}b)(e^2+f^2-\frac{1}{4}(a^2+b^2)-2\beta yg)+k^2(\cos
\frac{k}{2}a+\cos \frac{k}{2}b}\\
&{-\cos ke-\cos kf)\big] g^2}\\
=&{-\big[\frac{k^4}{6}(2+\cos \frac{k}{2}a+\cos
\frac{k}{2}b)(e^2+f^2-\frac{1}{4}(a^2+b^2)-2\beta yg)+k^2(\cos \frac{k}{2}a+\cos
\frac{k}{2}b}\\
&{-\cos ke-\cos kf)\big]y^2+k^2(2+\cos \frac{k}{2}a+\cos
\frac{k}{2}b)(e^2+f^2-\frac{1}{4}(a^2+b^2)-2\beta yg)}\\
&{+4(1+\cos
\frac{k}{2}a)(1+\cos \frac{k}{2}b)-4(1+\cos ke)(1+\cos kf)+O(y^2g^2)+O(y^4+g^4)\,.}
\end{align*}

\noindent Adding up the above two equations then yields
\begin{equation} \label{25}
\aligned &{ \big[\frac{k^4}{6}(2+\cos \frac{k}{2}a+\cos
\frac{k}{2}b)(c^2+d^2+e^2+f^2-\frac{1}{2}(a^2+b^2)-2\alpha xg-2\beta yg)\notag}\\
&{+k^2(2\cos\frac{k}{2}a+2\cos \frac{k}{2}b-\cos kc-\cos kd-\cos ke-\cos kf)\big] g^2}\\
=&{-\big[\frac{k^4}{6}(2+\cos \frac{k}{2}a+\cos
\frac{k}{2}b)(c^2+d^2-\frac{1}{4}(a^2+b^2)-2\alpha xg)+k^2(\cos \frac{k}{2}a+\cos
\frac{k}{2}b}\\
&{-\cos kc-\cos kd)\big]x^2}\\
&{-\big[\frac{k^4}{6}(2+\cos \frac{k}{2}a+\cos
\frac{k}{2}b)(e^2+f^2-\frac{1}{4}(a^2+b^2)-2\beta yg)+k^2(\cos \frac{k}{2}a+\cos
\frac{k}{2}b}\\
&{-\cos ke-\cos kf)\big]y^2}\\
&{+k^2(2+\cos \frac{k}{2}a+\cos
\frac{k}{2}b)(c^2+d^2+e^2+f^2-\frac{1}{2}(a^2+b^2)-2\alpha xg-2\beta yg)}\\
&{+8(1+\cos
\frac{k}{2}a)(1+\cos \frac{k}{2}b)-4(1+\cos kc)(1+\cos kd)-4(1+\cos ke)(1+\cos
kf)}\\
&{+O(x^2g^2)+O(y^2g^2)+O(g^4+x^4+y^4)\,.}
\endaligned
\end{equation}
\vskip 1mm \noindent The lemma follows immediately by rearranging
the terms above and using the definitions of $W,R,S,T,U$.
\end{proof}

\begin{rem} \label{67}For fixed $a$, as $x,y\to 0$ (and thus $g\to 0, b\to
a, \text{ and }c,d,e,f\to \frac{1}{2}a$), we have $ \frac{S}{W}\to 1
\text{  and  } \frac{T}{W}\to 1.$\end{rem}

\begin{lem} \label{31} For $a,b,x,y, k$ small enough, we have
$$U-Rg^2\geq \frac{k^2}{2}\,g^2-4k^2(2|\alpha| xg+ 2|\beta| yg)\,.$$\end{lem}

\begin{proof}  Let $|\cdot|$ denote the Euclidean distance in
$\mathbf{R}^{3}$ and $\angle EF'D=\theta$ (see Figure \ref{21}).
Then
$$|FF'|=\frac{1}{k}(1-\cos\frac{k}{2}b), \quad
|FF''|=\frac{1}{k}(1-\cos kg)=\frac{2}{k}\sin^2\left(\frac{k}{2}g\right),\quad |EF''|=\frac{1}{k}\sin
kg,$$ and
$$|EF'|^2=|F'F''|^2+|EF''|^2=\frac{1}{k^2}\left(1-\cos\frac{k}{2}b-2\sin^2\left(\frac{k}{2}g\right)\right)^2
+\frac{1}{k^2}\sin^2(kg).$$
\noindent Thus,
\begin{align} \label{68}
&{|CE|^2+|DE|^2-(|CF|^2+|DF|^2)\notag}\\
=&{\,(|CF'|+|EF'|\cos \theta)^2+(|EF'|\sin
\theta)^2\notag}\\
&{+(|CF'|-|EF'|\cos \theta)^2+(|EF'|\sin \theta)^2-2|CF'|^2-2|FF'|^2\notag}\\
=&{\,2(|EF'|^2-|FF'|^2)}\\
=&{\frac{2}{k^2}\left[\left(1-\cos\frac{k}{2}b-2\sin^2\left(\frac{k}{2}g\right)\right)^2+\sin^2(kg)-\left(1-\cos\frac{k}{2}b\right)^2\right]\notag}\\
=&{\,\frac{2}{k^2}\left[\sin^2(kg)-4\left(1-\cos\frac{k}{2}b\right)\sin^2\left(\frac{k}{2}g\right)+4\sin^4\left(\frac{k}{2}g\right)\right]\notag}\\
=&{\,2\left(\cos\frac{k}{2}b\right)g^2+O(g^4)\,>\,\left(\cos\frac{k}{2}b\right)\,g^2\,,\notag}
\end{align}
for $b,\, g$ small, which also implies $|EF'|\geq
|FF'|\,.$\,\,\vskip 2mm

Similarly, for $n\geq 2,$
\begin{align} \label{69}
&{|CE|^{2n}+|DE|^{2n}-(|CF|^{2n}+|DF|^{2n})\notag}\\
=&{\,\left[(|CF'|+|EF'|\cos
\theta)^2+(|EF'|\sin \theta)^2\right]^n\notag}\\
&{\,+\left[(|CF'|-|EF'|\cos \theta)^2+(|EF'|\sin \theta)^2\right]^n-2\left[|CF'|^2+|FF'|^2\right]^n}\\
\geq &{\,2\left[(|CF'|^2+|EF'|^2)^n-(|CF'|^2+|FF'|^2)^n\right]\,\geq
\,0.\notag}
\end{align}
\input{figure04.TpX}
\noindent  Now, note that
 $$c=\frac{2}{k}\arcsin \left(\frac{k|DE|}{2}\right)\,,\, f=\frac{2}{k}\arcsin \left(\frac{k|CE|}{2}\right)\,, $$
 and
 $$ \frac{1}{2}b=\frac{2}{k}\arcsin \left(\frac{k|CF|}{2}\right)=\frac{2}{k}\arcsin
 \left(\frac{k|DF|}{2}\right)\,.$$

\noindent  The Taylor series expansion
$$
\left(2\arcsin\left(\frac{x}{2}\right)\right)^2=x^2+\sum_{n=2}C_nx^{2n}\quad
(C_n\geq 0)$$ and (\ref{68}), (\ref{69}) then imply: for $x,y$ small
enough (thus $g$ is small enough)\,,
\begin{align} &c^2+f^2-\frac{1}{2}b^2\notag\\
=&\,|CE|^2+|DE|^2-(|CF|^2+|DF|^2)+\sum_{n=2}C_nk^{2n-2}\,\left(|CE|^{2n}+|DE|^{2n}-(|CF|^{2n}+|DF|^{2n})\right)\notag\\
>&\,\left(\cos\frac{k}{2}b\right)\,g^2\,.\notag
\end{align}

\noindent Similarly, \quad $d^2+e^2-\frac{1}{2}a^2>
(\cos\frac{k}{2}a)\,g^2\,.$\\

\noindent Therefore,
\begin{equation} \label{70}
c^2+d^2+e^2+f^2-\frac{1}{2}(a^2+b^2)> \left(\cos \frac{k}{2}a+\cos
\frac{k}{2}b\right)\,g^2.
\end{equation}

\noindent Recall the facts that as $x,y\to 0$\vskip 1mm
\begin{itemize}
\item[(1)] $c^2+d^2+e^2+f^2-\frac{1}{3}(a^2+b^2)\,\, \to\,\,\frac{1}{6}(a^2+b^2)$\,,\vskip 2mm
\item[(2)] $\cos kc+\cos kd+\cos ke+\cos kf\,\, \to\,\,
2(\cos\frac{k}{2}a+\cos\frac{k}{2}b)$\,.\vskip 2mm
\end{itemize}
\noindent One observes that in a small geodesic ball $B_{\tau}$
(thus $a, b, x, y$ are small enough), we have:\vskip 2mm
\begin{itemize}
\item[(1)] $\cos\frac{k}{2}a+\cos\frac{k}{2}b\,>\,\frac{15}{8}$\,,\vskip 2mm
\item[(2)]
$c^2+d^2+e^2+f^2-\frac{1}{3}(a^2+b^2)-2\alpha xg-2\beta yg\,<\,\frac{1}{2}$\,,\vskip 2mm
\item[(3)] $\cos kc+\cos kd+\cos ke+\cos
kf\,>\,\frac{3}{2}\left(\cos\frac{k}{2}a+\cos\frac{k}{2}b\right)\,.$\vskip
2mm
\end{itemize}

\noindent Therefore, using (\ref{70}) and the definition of $R$, we
obtain that for $a,b,x,y$ small, $k\leq 1$ :
\begin{align*}
 &{U-Rg^2}\\
 =&{\,k^2\left(2+\cos\frac{k}{2}a+\cos
\frac{k}{2}b\right)\left(c^2+d^2+e^2+f^2-\frac{1}{2}(a^2+b^2)-2\alpha xg- 2\beta yg\right)-Rg^2}\\
\geq&{\,\frac{15k^2}{8}\left(2+\cos\frac{k}{2}a+\cos
\frac{k}{2}b\right)g^2-\left(\frac{k^4}{4}\left(2+\cos\frac{k}{2}a+\cos
\frac{k}{2}b\right)+3k^2\left(\cos\frac{k}{2}a+\cos
\frac{k}{2}b\right)\right)g^2}\\
&{-k^2\left(2+\cos\frac{k}{2}a+\cos
\frac{k}{2}b\right)(2\alpha xg+ 2\beta yg)}\\
\geq &{k^2\left(\frac{13}{4}-\frac{11}{8}\left(\cos\frac{k}{2}a+\cos
\frac{k}{2}b\right)\right)g^2-4k^2(2|\alpha| xg+ 2|\beta| yg)}\\
\geq &{\,\frac{k^2}{2}\,g^2-4k^2(2|\alpha| xg+ 2|\beta| yg)\,.\notag}\end{align*}
\end{proof}

\begin{lem}  $ V\geq 0$ for $a,b,x,y$ small enough. \end{lem}

\begin{proof}  By the triangle inequality, we know $\frac{1}{2}(a+b)<c+d<\frac{1}{2}(a+b)+x+g\,, \frac{1}{2}(a+b)<e+f<\frac{1}{2}(a+b)+y+g\,,c+f>b$\,,
and $d+e>a$\,. If $c\geq\frac{1}{2}b\,,\,f\geq\frac{1}{2}b$ and $
d\geq\frac{1}{2}a\,,\,e\geq\frac{1}{2}a$ then one easily sees $V\geq
0.$ Now without loss of generality we suppose that there exists
$\varsigma>\sigma>0$ such that \begin{align*}c=\frac{1}{2}b-\sigma,
\quad d=\frac{1}{2}a+\varsigma,\quad e>\frac{1}{2}a-\varsigma,\quad
f>\frac{1}{2}b+\sigma .\end{align*}

\noindent Assume now $k=1$, then for $a,b$ small:
\begin{align*} V=&{8(\cos\frac{1}{2}a+1)(\cos \frac{1}{2}b+1)-4(\cos
c+1)(\cos d+1)-4(\cos e+1)(\cos f+1)}\\
=&{8\cos\frac{1}{2}b-4(\cos c+\cos f)+8\cos\frac{1}{2}a-4(\cos
d+\cos e)+8(\cos \frac{1}{2}a)(\cos \frac{1}{2}b)}\\
&{-4(\cos c)(\cos d)-4(\cos
e)(\cos f)}\\
 \geq
&{8\cos\frac{1}{2}b-4(\cos(\frac{1}{2}b-\sigma)+\cos(\frac{1}{2}b+\sigma))+8\cos\frac{1}{2}a-4(\cos(\frac{1}{2}a+\varsigma)+\cos(\frac{1}{2}a-\varsigma))}\\
&{+8(\cos \frac{1}{2}a)(\cos \frac{1}{2}b)
-4\left[\cos(\frac{1}{2}b-\sigma)\cos(\frac{1}{2}a+\varsigma)+\cos(\frac{1}{2}b+\sigma)\cos(\frac{1}{2}a-\varsigma)\right]}\\
=&{8(1-\cos\sigma)\cos\frac{1}{2}b+8(1-\cos\varsigma)\cos\frac{1}{2}a+4(\cos\frac{a+b}{2}+\cos\frac{a-b}{2})}\\
&{-2\left[\cos(\frac{a+b}{2}+\varsigma-\sigma)+\cos(\frac{b-a}{2}+\varsigma+\sigma)+\cos(\frac{a+b}{2}-\varsigma+\sigma)+\cos(\frac{b-a}{2}-\varsigma-\sigma)\right]}\\
\geq&{4(1-\cos(\varsigma-\sigma))\cos\frac{a+b}{2}+4(1-\cos(\varsigma+\sigma))\cos\frac{b-a}{2}\,\geq\,0\,.}
 \end{align*}
 One sees that the above argument works for all $k>0$, completing the proof.
\end{proof}

\begin{rem} From the proofs of these lemmas we also see that for $a,b,x,y$
small enough, $W,S,T,R>0.$ \end{rem}

\begin{proof} (of {\bf{Lemma \ref{34}}}) Combining previous lemmas we get
$$ W(x^2+y^2-2g^2)\geq Sx^2+Ty^2+\frac{k^2}{2}\,g^2-4k^2(2|\alpha| xg+ 2|\beta| yg)+O(x^2g^2)+O(y^2g^2)+O(g^4+x^4+y^4).$$
 Now using nothing but the facts\vskip 2mm
\begin{enumerate}
\item $0<k^2\leq W \leq 4k^2$ for $a,b,x,y, k$ sufficiently small\,,\vskip 2mm
\item $\frac{S}{W}\to 1$ \,and\, $ \frac{T}{W}\to 1$\,\,uniformly as $x,y\to
0$ \,\,(see Remark \ref{67})\,, \vskip 2mm
\item \label{29} $|a-b|\leq x+y$,\vskip 2mm
\end{enumerate}
we obtain for $a,b, x, y$ sufficiently small:
\begin{align}
& x^2+y^2-2g^2\geq \frac{3}{4}(x^2+y^2)+ \frac{1}{16}g^2-4(2|\alpha| xg+ 2|\beta| yg)\,.\end{align}
Now choose $K$ sufficiently small so that $\max\{|\alpha|, |\beta|\}\leq \frac{1}{128}$ and thus by Cauchy-Schwarz inequality
\begin{equation}
4(2|\alpha| xg+ 2|\beta| yg) \leq \frac{1}{16}(x^2+y^2+g^2)\,,\notag
\end{equation}
and therefore
\begin{equation}x^2+y^2-2g^2\geq \frac{1}{2}(x^2+y^2)\geq \frac{1}{4}(x+y)^2\geq \frac{1}{4}(a-b)^2\,,\notag
\end{equation}
completing the proof of Lemma \ref{34}.
\end{proof}

 \end{document}